\documentclass[11pt,a4paper]{article}
\usepackage{amsmath,amsthm}
\usepackage{amssymb}
\usepackage{graphicx}
\newtheorem{thm}{Theorem}[section]
\newtheorem{lem}{Lemma}[section]

\usepackage{color}
\usepackage{tikz}
\usepackage{enumerate}
\usetikzlibrary{patterns}
\author{Yasunari Hikima, Atsushi Iwasaki\thanks{Kyoto University\ \ e-mail: iwasaki.atsushi.4x@kyoto-u.ac.jp}, and Ken Umeno\thanks{Kyoto University\ \ e-mail: umeno.ken.8z@kyoto-u.ac.jp}}
\title{The reference distributions of Maurer's universal statistical test and its improved tests}
\begin{document}
\maketitle
\begin{abstract}
Maurer's universal statistical test can widely detect non-randomness of given sequences.
Coron proposed an improved test, and further Yamamoto and Liu proposed a new test based on Coron's test.
These tests use normal distributions as their reference distributions, but the soundness has not been theoretically discussed so far.
Additionally, Yamamoto and Liu's test uses an experimental value as the variance of its reference distribution.
In this paper, we theoretically derive the variance of the reference distribution of Yamamoto and Liu's test and prove that the true reference distribution of Coron's test converges to a normal distribution in some sense.
We can apply the proof to the other tests with small changes.
\end{abstract}
\section{Introduction}
Random number sequences are used in many fields such as the Monte Carlo method and information security, including cryptography.
In particular, for cryptography, high randomness is required.
Therefore, we have to evaluate the randomness of a given sequence or property of its generator.
Randomness test is one of such evaluating method.
It is a hypothesis test, and the null is that the given sequence is truly random.
It is just an experimental method and cannot give proof of randomness.
On the other hand, the randomness test has a merit that we can adapt the test for any sequences regardless of their generators.
For instance, Tamura and Shikano used randomness tests to inspect the properties of a quantum computer developed by IBM \cite{Tamura-Shikano}.

Since we can consider numberless alternative hypotheses against the null, many randomness tests have been proposed.
We do not use a single randomness test for evaluation in practice, but a set of randomness tests called a test suite.
There are many test suites such as TestU01, Diehard and Dieharder \cite{TestU01,Diehard,Dieharder}, and NIST SP800-22 \cite{NIST} is one of the most widely used test suites.
SP800-22 consists of 15 types of randomness tests.
Some of the tests were pointed out that are problematic, and improvements have been proposed \cite{Kim-Umeno-Hasegawa,Hamano,Pareschi,Okada,Iwasaki,Hamamo-Kaneko}.

This paper focuses on Maurer's universal statistical test \cite{Maurer}, which is included in SP800-22.
Almost all proposed randomness tests aim to detect non-randomness which the developers explicitly specify, but universal tests that aim to detect a wide range of non-randomness are exceptions.
Maurer's universal statistical test was also developed to have such property.
Coron proposed an improved test based on Maurer's test \cite{Coron}, but the improved test has not been adopted for SP800-22 so far.
Further, Yamamoto and Liu added an optional improvement to Coron's test and state that their test can detect non-randomness more sensitively \cite{Yamamoto-Liu}.
These three tests use normal distributions as their reference distributions.
As a common problem of the three tests, it has not been proven that their true reference distributions are normal distributions at least approximately.
In other words, these tests use normal distributions as their reference distributions without theoretical evidence, although there is no report that using normal distribution is improper as far as the authors know.
This is one of the problems we address in this paper.
Additionally, the expected values and variances of the normal distributions are needed to perform these tests.
Deriving the theoretical values of the variances is harder than deriving the expected values.
As the variance of the Maurer's test, an experimental value was at first used, and later, Coron and Naccache derived the theoretical value \cite{Coron-Naccache}.
The method in Ref. \cite{Coron-Naccache} can be applied to Coron's test and derived the theoretical value of the variance of Coron's test.
However, the theoretical value of the variance of Yamamoto and Liu's test has not been derived, and an experimental value is used.
The lack of the theoretical value is the other problem we address in this paper.

This paper is organized as follows.
In section 2, we introduce Maurer's test, Coron's test, and Yamamoto and Liu's test.
In section 3, we theoretically derive the variance of the reference distribution of Yamamoto and Liu's test.
Section 3 is a reconstruction of Ref. \cite{Hikima-Iwasaki-Umeno} and \cite{Hikima}.
In section 4, we prove that the distribution of the test statistic of Coron's test converges to a normal distribution in some sense.
The proof can be applied for Maurer's test and Yamamoto and Liu's test with minor change.
Finally, we offer a conclusion.

\section{Maurer's test, Coron's test, and Yamamoto and Liu's test}
This section briefly introduces Maurer's test, Coron's test, and Yamamoto and Liu's test.
\subsection{Maurer's universal statistical test}
Let us consider a case that we test a given $N$-bit sequence $X$.
First, we divide $X$ into $L$-bit blocks.
We use the first $Q$ blocks for initialization, and the remained $K$ blocks for testing.
Here, $Q$ and $K$ are required to be sufficiently large comparing with $2^L$.
In practical, $Q\geq10\cdot2^L$ and $K\geq1000\cdot2^L$ are suggested.
For simplicity, we assume that $N=L\times (Q+K)$.
Let $b_n(X)$ be the $n$-th block and define $A_n(X)$ as
\begin{align}
A_n(X)=
\begin{cases}
n\ &(\text{$b_{n-m}(X)\ne b_n(X)$ for $m=1,2,\dots,n-1$})\\
\min \{m\in\mathbb{N}\mid b_{n-m}(X)=b_n(X)\} &(\text{otherwise})
\end{cases}
.
\end{align}
In other words, $A_n(X)$ implies the distance between the $n$-th block and the nearest former block which has the same value of the $n$-th block if exists.
Using $A_n(X)$, the test statistic is described as
\begin{align}
f_M(X):=\frac{1}{K}\sum_{n=Q+1}^{Q+K}\log_2 A_n(X).
\end{align}
Finally, using a normal distribution as the reference distribution (i.e., assuming that $f_M\left(R^{(0.5)}\right)$ follows a normal distribution where $R^{(p)}$ is a random variable sequence and each element independently takes one with probability $p$ and zero with probability $1-p$), we compute the corresponding p-value.
In practical use of Maurer's test, we need the expected value and variance of the reference distribution.
For sufficiently large $Q$, we can regard that the distributions of $A_n\left(R^{(0.5)}\right)$ $(n\geq Q+1)$ do not depend on $n$ and their support is $\mathbb{N}$.
Then, the expected value is given as
\begin{align}
\mathbb{E}\left[f_M\left(R^{(0.5)}\right)\right]=&\frac{1}{K}\sum_{n=Q+1}^{Q+K}\mathbb{E}\left[\log A_n\left(R^{(0.5)}\right)\right]\\
\simeq& 2^{-L}\sum_{l=1}^\infty(1-2^{-L})^{l-1}\log l.
\label{Maurer-heikin}
\end{align}
As an experimentally derived approximation of the variance, Maurer proposed that
\begin{align}
\text{Var}\left[f_M\left(R^{(0.5)}\right)\right]\simeq \left\{0.7-\frac{0.8}{L}+\left(1.6+\frac{12.8}{L}\right)K^{-\frac{4}{L}}\right\}\frac{\text{Var}\left[\log A_n\left(R^{(0.5)}\right)\right]}{K},
\end{align}
where by the same reason of deriving (\ref{Maurer-heikin}),
\begin{align}
\text{Var}\left[\log A_n\left(R^{(0.5)}\right)\right]\simeq 2^{-L}\sum_{l=1}^\infty(1-2^{-L})^{l-1}\left(\log l\right)^2-\left\{2^{-L}\sum_{l=1}^\infty(1-2^{-L})^{l-1}\log l\right\}^2.
\end{align}
Later, Coron and Naccache derived the theoretical formula of the variance of the reference distribution for $Q\to\infty$ \cite{Coron-Naccache}.

The test statistic of Maurer's test relates to entropy of the tested sequence and its universality is stated based on the fact.
Maurer proved that
\begin{align}
\lim_{L\to\infty}\left\{\mathbb{E}\left[f_M\left(R^{(p)}\right)\right]-LH(p)\right\}=C
\end{align}
where $C$ is a constant and $H(p):=-p\ln p-(1-p)\ln (1-p)$ is per bit entropy of $R^{(p)}$.
It was also proven that
\begin{align}
\lim_{L\to\infty}\frac{\mathbb{E}\left[f_M\left(U_{\mathcal{S}}\right)\right]}{L}=H_{\mathcal{S}},
\end{align}
where $U_{\mathcal{S}}$ is a binary sequence generated by an ergodic stationary source $\mathcal{S}$ which has finite memory and $H_{\mathcal{S}}$ is per bit entropy $\mathcal{S}$ generates.
Maurer gave a conjecture that
\begin{align}
\lim_{L\to\infty}\left\{\mathbb{E}\left[f_M\left(U_{\mathcal{S}}\right)\right]-LH_{\mathcal{S}}\right\}=C,
\end{align}
but it was shown that the conjecture is false \cite{Coron-Naccache}.

\subsection{Coron's universal statistical test}
Coron proposed a test which replaces the test statistic $f_M$ by
\begin{align}
f_C(X):=\frac{1}{K}\sum_{n=Q+1}^{Q+K}g\left(A_n(X)\right),
\end{align}
where
\begin{align}
g(i)=\frac{1}{\ln 2}\sum_{j=1}^{i-1}\frac{1}{j}.
\end{align}
This test uses a normal distribution as its reference distribution, too.
For sufficiently large $Q$, the expected value of $f_C\left(R^{(p)}\right)$ is approximated as
\begin{align}
\mathbb{E}\left[f_C\left(R^{(p)}\right)\right]=&\frac{1}{K}\sum_{n=Q+1}^{Q+K}\mathbb{E}\left[g\left(A_n\left(R^{(p)}\right)\right)\right]\\
\simeq&L H(p)\label{Coron-theorem}
\end{align}
for any $p\in(0,1)$.
Then, the expected value of the reference distribution is given as $L H(0.5)$.
For $p=\frac{1}{2}$, by the same way in Ref. \cite{Coron-Naccache}, we can derive the theoretical formula of $\text{Var}\left[f_C\left(R^{(0.5)}\right)\right]$ as $Q\to\infty$.
Based on the theoretical formula, Coron proposed an approximation
\begin{align}\text{Var}\left[f_C\left(R^{(0.5)}\right)\right]\simeq\left\{d(L)+\frac{e(L)\times2^L}{K}\right\}\frac{\text{Var}\left[g\left( A_n\left(R^{(0.5)}\right)\right)\right]}{K}
\end{align}
for sufficient large $Q$, and $d(L)$ and $e(L)$ are given in Ref. \cite{Coron}.

As seen in (\ref{Coron-theorem}), the test statistic of Coron's test relates to entropy of the tested sequence, too.
Unlike the case of Maurer's test, (\ref{Coron-theorem}) holds for finite $L$ and this is an advantage of Coron's test.

\subsection{Yamamoto and Liu's test}
Yamamoto and Liu proposed to flip a part of bits in a given sequence before performing Coron's test \cite{Yamamoto-Liu}.
The flipping is stochastic, and each bit $x_k$ is independently converted to $\hat{x}_k$ following the rule as
\begin{align}
\text{Pr}\left[\hat{x}_k=0|x_k=0\right]=&1,\label{eq:flip1}\\
\text{Pr}\left[\hat{x}_k=1|x_k=1\right]=&\alpha.\label{eq:flip2}
\end{align}
They stated that the flipping made Coron's test more sensitive, and showed some experimental results to detect non-randomness sensitively.
It is obvious that the exact reference distribution is the distribution of $f_C\left(R^{(p)}\right)$ for $p=\frac{\alpha}{2}$.
Approximately, this test also uses a normal distribution as its reference distribution, but the expected value and variance are not equal to those of Coron's test.
As we mentioned in the former subsection, (\ref{Coron-theorem}) holds for any $p\in(0,1)$ and we can use (\ref{Coron-theorem}) with $p=\frac{\alpha}{2}$ as the expected value for sufficiently large $Q$.
On the other hand, an experimentally derived value is used as the variance.
Yamamoto and Liu pointed out that $\left\{A_n\left(R^{(p)}\right)\right\}$ are not i.i.d, and thus to derive the theoretical variance is difficult.
However, the derivation process for the variances of the reference distributions of Maurer's test and Coron's test do not need independent but just identity.
Even in Yamamoto and Liu's test, identity is preserved.
Thus, we expect that we can theoretically derive the variance of the reference distribution of Yamamoto and Liu's test by a similar way of the derivation process.

\section{Derivation of the variance of the reference distribution of Yamamoto and Liu's test}
In this section, we derive the theoretical formula of the variance of the reference distribution that Yamamoto and Liu's test uses and numerically evaluate the formula.
To get the formula, we need the distributions of $A_n\left(R^{(p)}\right)$ and the joint distributions of $\left(A_n\left(R^{(p)}\right),A_{n+k}\left(R^{(p)}\right)\right)$.
For simplicity, we omit to write argument $R^{(p)}$ for every variable.
In the following, we replace the index of $A_n$ as illustrated in Figure \ref{fig:replace} and consider the case that $Q\to\infty$.
Then, the sequence of $\{A_n\}_{n=1}^{K}$ exactly follows a stationary ergodic process, that is, the joint distribution of $\{A_n\}_{n=k}^{k+m}$ depends only on $m$, and its support fully covers $\mathbb{N}^m$.
\begin{figure}[b]
\centering
\includegraphics{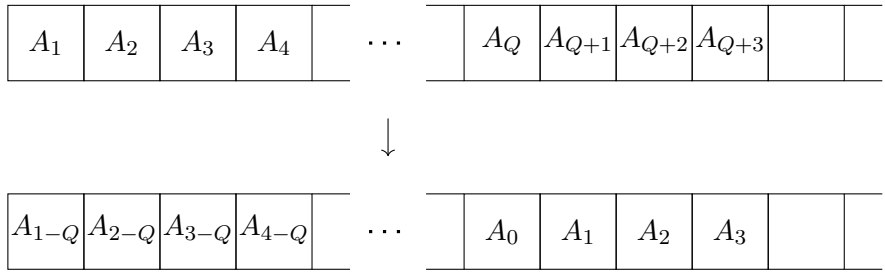}
\caption{Replacement of the index}
\label{fig:replace}
\end{figure}
\subsection{The distribution of $A_n$}
We consider the event of $\mathcal{M}:=\left< A_n=i \right>$ for $i\geq 1$.
As illustrated in Figure \ref{fig:A_n=i}, $\mathcal{M}$ is equivalent to the $n$-th block coincides $(n-i)$-th block and do not coincide other blocks between $n$-th and $(n-i)$-th blocks, i.e.,
\begin{align}
\mathcal{M} = \left< b_{n-i} = b_{n}, b_{n-i+1} \neq b_{n} , \dots, b_{n-1} \neq b_{n} \right>.
\end{align}
Let $\ell(b)$ be the number of ``$1$'' included in a block $b \in \{0,1\}^L$.
We have
\begin{align}
\mathrm{Pr}[\mathcal{M} \mid \ell(b_n) = r] &= w_r \times ( 1 - w_r )^{i-1},\label{eq:probability_M_mid} \\
\mathrm{Pr} [\ell(b_n) = r] &= \binom{L}{r} w_r, \label{eq:probability_l_r}
\end{align}
where $w_r = p^r (1-p)^{L-r}$ and $\binom{L}{r} = \frac{L!}{r!(L-r)!}$ is a binomial coefficient.
Then, we obtain
\begin{align}
\mathrm{Pr}[A_n=i] =& \sum_{r=0}^{L} \mathrm{Pr}[\mathcal{M} \mid \ell(b_n) = r] \times \mathrm{Pr} [\ell(b_n) = r],\\
=& \sum_{r=0}^{L} \binom{L}{r} w_r^2 ( 1 - w_r )^{i-1}\label{A_n=i}.
\end{align}
\begin{figure}
\centering
\includegraphics{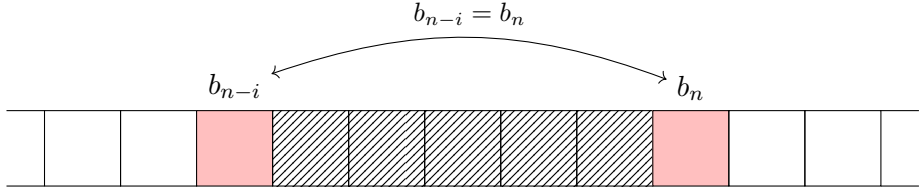}
\caption{The arrangement of blocks in the case of $A_n=i$.}
\label{fig:A_n=i}
\end{figure}
%
\subsection{The joint distribution of $A_n$ and $A_{n+k}$}\label{subsec:4-2}
We derive the probability of the event $\left< A_n=i,\, A_{n+k}=j \right>$ for $i\geq 1$ and $j\geq 1$.
Depending on $k$, $i$ and $j$, we consider five cases.
\subsubsection{Case 1: $1 \leq j \leq k-1$}
When $1 \leq j \leq k-1$, there is no overlapping between $(n-i)$-th to $n$-th blocks and $(n+k-j)$-th to $(n+k)$-th blocks as illustrated in Figure \ref{fig:case1}.
Then, we obtain
\begin{align}
\mathrm{Pr}[A_n=i, A_{n+k}=j] =& \mathrm{Pr}[A_n=i] \times \mathrm{Pr}[A_{n+k}=j]\\
=&\left( \sum_{r=0}^{L} \binom{L}{r}w_r^2 (1-w_r)^{i-1} \right) \times \left( \sum_{r=0}^{L} \binom{L}{r}w_r^2 (1-w_r)^{j-1} \right).\label{eq:joint1}
\end{align}
\begin{figure}
\centering
\includegraphics{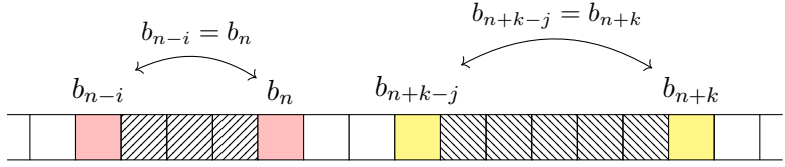}
\caption{An example of the arrangement of blocks in the case of $1\leq j \leq k-1$}
\label{fig:case1}
\end{figure}
\subsubsection{Case 2: $j=k$}
We fix $b \in \{0,1\}^{L}$, and consider the event $e_2(b):=\left< A_n=i ,\, A_{n+k}=j,\,b_n=b\right>$ for $j=k$ as illustrated in Figure \ref{fig:case2}.
The event $e_2(b)$ can be written as
\begin{align}\label{eq:e_2}
\begin{split}
e_2(b)
= &\left< b_{n-i} = b , b_{n} = b, b_{n+k} = b \right> \\
&\land \left< b_{n-i+1} \neq b , \dots , b_{n-1} \neq b \right> \\
&\land \left< b_{n+1} \neq b , \dots , b_{n+k-1} \neq b \right>.
\end{split}
\end{align}
Then, we have
\begin{align}\label{eq:probability_e_2}
\mathrm{Pr} \left[ e_2(b) \right]
=w_r^3 \times (1-w_{r})^{i+k-2},
\end{align}
where $r=l(b)$.
If $b\ne b^\prime$, then it is obvious that
\begin{align}
e_2(b)\land e_2(b^\prime)=\emptyset.
\end{align}
Then, since $\left< A_n=i ,\, A_{n+k}=j\right> = \bigvee_{b\in B^L} e_2(b)$ and the r.h.s. of (\ref{eq:probability_e_2}) depends only on $l(b)$, we obtain
\begin{align}
\mathrm{Pr}[A_n=i ,\, A_{n+k}=j] &=\mathrm{Pr}\left[ \bigvee_{b \in B^L} e_2(b) \right] \\
&= \sum_{b\in B^L} \mathrm{Pr} [e_2(b)] \\
&= \sum_{b\in B_0^L \cup \dots \cup B_L^L} \mathrm{Pr} [e_2(b)] \\
&= \sum_{r=0}^{L} \sum_{b\in B_{r}^L} \mathrm{Pr} \left[ e_2(b) \right] \\
&= \sum_{r=0}^{L} \dbinom{L}{r} w_{r}^3 (1-w_{r})^{i+k-2},\label{eq:joint2}
\end{align}
where $B^L=\{0,1\}^L$ and $B_r^L = \{ b\in B^L \mid \ell(b)=r \}$.
\begin{figure}
\centering
\includegraphics{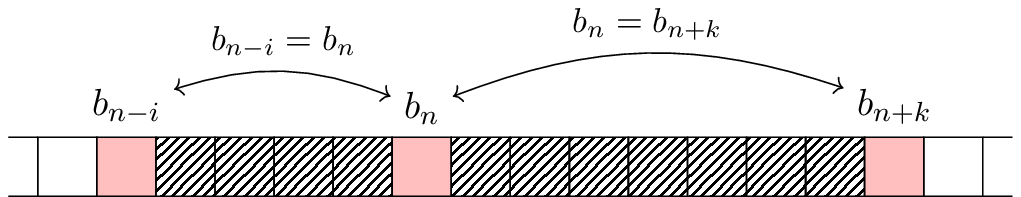}
\caption{An example of the arrangement of blocks in the case of $j = k$}
\label{fig:case2}
\end{figure}
%
\subsubsection{Case 3: $k+1 \leq j \leq k+i-1$}
We fix $b,\, b^\prime \in B^L$, and consider the event $e_3(b,b^\prime):= \left< A_n=i ,\, A_{n+k}=j,\,b_n=b,\,b_{n+k}=b^\prime\right>$ for $k+1 \leq j \leq k+i-1$ as illustrated in Figure \ref{fig:case3}.
The event $e_3(b,b^\prime)$ can be written as
\begin{align}
\begin{split}
\label{eq:e_3}
e_3 \left(b,b^\prime\right) =
&\left< b_{n-i} = b , b_{n} = b , b_{n+k-j} = b^\prime , b_{n+k} = b^\prime  , b_{n+k} \ne b_n \right> \\
&\land \left< b_{n-i+1} \neq b, \dots, b_{n+k-j-1} \neq b \right> \\
&\land \left< b_{n+k-j+1} \neq b, \dots, b_{n-1} \neq b \right> \\
&\land \left< b_{n+k-j+1} \neq b^\prime, \dots, b_{n-1} \neq b^\prime \right> \\
&\land \left< b_{n+1} \neq b^\prime , \dots, b_{n+k-1} \neq b^\prime \right>.
\end{split}
\end{align}
If $b=b^\prime$, then it is obvious that
\begin{align}
\mathrm{Pr} \left[ e_3 \left(b,b^\prime\right) \right] =0.
\end{align}
If $b\ne b^\prime$, then we have
\begin{align}
\label{eq:probability_e3}
\mathrm{Pr} \left[ e_3 \left(b,b^\prime\right) \right]
=& w_{r_1}^2 w_{r_2}^2
(1-w_{r_1})^{i-j+k-1}
(1-w_{r_1}-w_{r_2})^{j-k-1}
(1-w_{r_2})^{k-1},
\end{align}
where $r_1=l(b)$ and $r_2=l(b^\prime)$.
Let $h_3(r_1,r_2)$ be the r.h.s. of (\ref{eq:probability_e3}).
Since $e_3(b,b^\prime)\land e_3(\tilde{b},\tilde{b^\prime})=\emptyset$ for $(b,b^\prime)\ne (\tilde{b},\tilde{b^\prime})$, we obtain
\begin{align}
\mathrm{Pr} [A_n=i ,\, A_{n+k}=j]
=&\sum_{b \in B^L} \sum_{b^\prime \in B^{L} \setminus \{ b \}} \mathrm{Pr} \left[ e_3(b,b^\prime) \right] \\
=& \sum_{r_1=0}^{L} \sum_{b \in B^L_{r_1}} \sum_{r_2=0}^{L} \sum_{b^\prime \in B^L_{r_2} \setminus \{ b \}} \mathrm{Pr} \left[ e_3(b,b^\prime) \right] \\
\begin{split}
=& \sum_{r_1=0}^{L} \sum_{r_2 \neq r_1} \sum_{b \in B^L_{r_1}} \sum_{b^\prime \in B^L_{r_2}} \mathrm{Pr} \left[e_3(b,b^\prime) \right]
\\&+ \sum_{r_1=0}^{L} \sum_{r_2 \in \{r_1\}} \sum_{b \in B^L_{r_1}} \sum_{b^\prime \in B^L_{r_2} \setminus \{ b_1 \}} \mathrm{Pr} \left[e_3(b,b^\prime) \right]
\end{split} \\
=& \sum_{r_1=0}^{L} \sum_{r_2 \neq r_1} \dbinom{L}{r_1} \dbinom{L}{r_2} h_3(r_1,r_2)
+ \sum_{r=0}^{L} \dbinom{L}{r} \left\{ \dbinom{L}{r} -1 \right\} h_3(r,r)\label{eq:joint3}.
\end{align}
\begin{figure}
\centering
\includegraphics{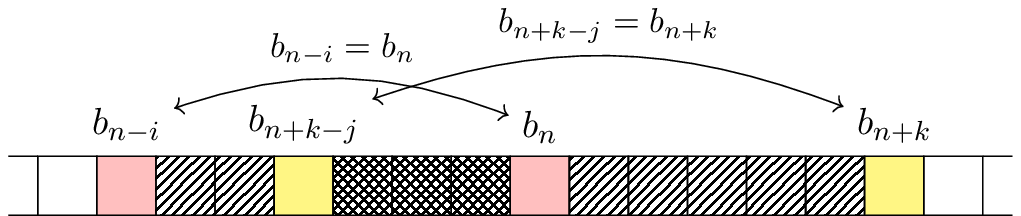}
\caption{An example of the arrangement of blocks in the case of $k+1 \leq j \leq k+i-1$}
\label{fig:case3}
\end{figure}
%
\subsubsection{Case 4: $j=k+i$}
As illustrated in Figure \ref{fig:case4}, it is obvious that
\begin{align}
\left< A_n=i \right>\land\left< A_{n+k}=j \right>=\emptyset
\end{align}
when $j=k+i$.
Then, we have
\begin{align}
\mathrm{Pr}[A_n=i,A_{n+k}=j] = 0.\label{eq:joint4}
\end{align}
\begin{figure}
\centering
\includegraphics{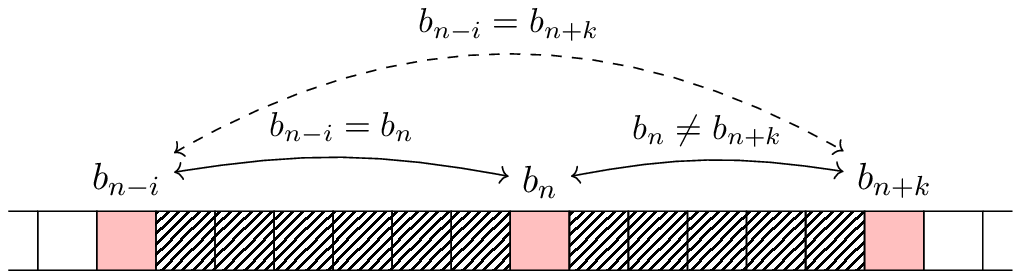}
\caption{An example of the arrangement of blocks in the case of $j=k+i$}
\label{fig:case4}
\end{figure}
%
\subsubsection{Case 5: $j \geq k+i+1$}
We fix $b,\, b^\prime \in B^L$, and consider the event $e_5(b,b^\prime):=\left< A_n=i ,\, A_{n+k}=j,b_n=b,\,b_{n+k}=b^\prime\right>$ for $j \geq k+i+1$ as illustrated in Figure \ref{fig:case5}.
The event $e_5(b,b^\prime)$ is written as
\begin{align}
\begin{split}
\label{eq:e_5}
e_5 (b,b^\prime) :=
&\left< b_{n+k-j} = b , b_{n-i} = b^\prime , b_{n} = b^\prime , b_{n+k} = b  , b_{n+k} \ne b_n\right> \\
&\land \left< b_{n+k-j+1} \neq b, \dots, b_{n-i-1} \neq b \right> \\
&\land \left< b_{n-i+1} \neq b, \dots, b_{n-1} \neq b \right> \\
&\land \left< b_{n-i+1} \neq b^\prime, \dots, b_{n-1} \neq b^\prime \right> \\
&\land \left< b_{n+1} \neq b , \dots, b_{n+k-1} \neq b \right>.
\end{split}
\end{align}
If $b=b^\prime$, then it is obvious that
\begin{align}
\mathrm{Pr} \left[ e_5 \left(b,b^\prime\right) \right] =0.
\end{align}
If $b\ne b^\prime$, then we have
\begin{align}
\label{eq:probability_e5}
\mathrm{Pr} \left[ e_5(b,b^\prime) \right]
=& w_{r_1}^2 w_{r_2}^2
(1-w_{r_1})^{-i+j-k-1}
(1-w_{r_1}-w_{r_2})^{i-1}
(1-w_{r_2})^{k-1},
\end{align}where $r_1=l(b)$ and $r_2=l(b^\prime)$.
Let $h_5(r_1,r_2)$ be the r.h.s. of (\ref{eq:probability_e3}).
Since $e_5(b,b^\prime)\land e_5(\tilde{b},\tilde{b^\prime})=\emptyset$ for $(b,b^\prime)\ne (\tilde{b},\tilde{b^\prime})$, by the same way in case 3, we obtain
\begin{align}
\mathrm{Pr} [A_n=i,\, A_{n+k}=j]
&= \sum_{r_1=0}^{L} \sum_{r_2 \neq r_1} \dbinom{L}{r_1} \dbinom{L}{r_2} h_5(r_1,r_2)
+ \sum_{r=0}^{L} \dbinom{L}{r} \left\{ \dbinom{L}{r} -1 \right\} h_5(r,r).\label{eq:joint5}
\end{align}
\begin{figure}
\centering
\includegraphics{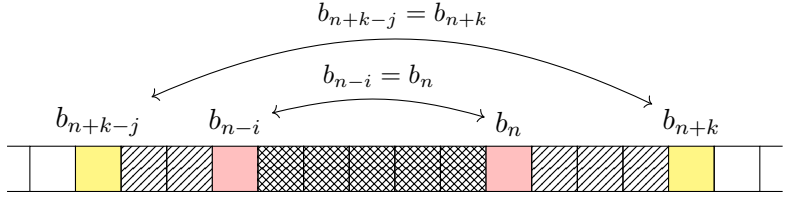}
\caption{An example of the arrangement of blocks in the case of $j \geq k+i+1$}
\label{fig:case5}
\end{figure}
%
\subsection{Derivation of the variance}
Finally, we derive $\mathrm{Var} [f_C]$ using $\mathrm{Pr}[A_n=i]$ and $\mathrm{Pr}[A_n=i, \, A_{n+k}=j]$.
Since $\{A_n\}_{n=1}^{K}$ is ergodic stationary under the assumption that $Q\to\infty$, $ \mathrm{Var} [g(A_n)]$ does not depend on $n$ and $\mathrm{Cov} [g(A_{i}), g(A_{j})]$ depends only on $j-i$.
Then, we get
\begin{align}
\label{eq:var_fC}
\mathrm{Var} [f_C]
=& \mathrm{Var} \left[ \frac{1}{K} \sum_{n=1}^{K} g(A_n) \right] \\
=& \frac{1}{K^2} \left( \sum_{n=1}^{K} \mathrm{Var} [g(A_n)] + 2 \sum_{1 \leq i < j \leq K} \mathrm{Cov} [g(A_{i}), g(A_{j})] \right)\\
=& \frac{1}{K}\mathrm{Var} [g(A_n)] +\frac{2}{K}\sum_{k=1}^{K-1}\left(1-\frac{k}{K}\right) \mathrm{Cov}[g(A_n),g(A_{n+k})]\label{final_var}.
\end{align}
Here, by (\ref{Coron-theorem}), we have
\begin{align}
\mathrm{Var}[g(A_n)]
&=\sum_{i=1}^{\infty} \{g(i)\}^2 \mathrm{Pr}[A_n=i]- \left\{LH(p)\right\}^2,\label{eq:var_g_def}\\
\mathrm{Cov}[g(A_n),g(A_{n+k})]
&= \sum_{i=1}^{\infty}\sum_{j=1}^{\infty}g(i)g(j)\mathrm{Pr}[A_n=i, \, A_{n+k}=j] - \left\{LH(p)\right\}^2\label{eq:covariance_g_g}.
\end{align}
Substituting (\ref{A_n=i}) into (\ref{eq:var_g_def}), and (\ref{eq:joint1}), (\ref{eq:joint2}), (\ref{eq:joint3}), (\ref{eq:joint4}) and (\ref{eq:joint5}) into (\ref{eq:covariance_g_g}), we can compute $\mathrm{Var}[g(A_n)]$ and $\mathrm{Cov}[g(A_n),g(A_{n+k})]$.
Then, using (\ref{final_var}), we obtain $\mathrm{Var} [f_C]$.

\subsection{Numerical analysis}\label{subsec:numerical_exp_L4}
We numerically analyzed $\mathrm{Var} [f_C]$ based on the above theoretical analysis.

\subsubsection{Experiment 1}
We computed $\mathrm{Var} [f_C]$ for $p=0.33,\, 0.4$ and $0.5$, where $L=4$.
Note that $p=0.33$ is suggested in Ref. \cite{Yamamoto-Liu}.
Figure \ref{fig:1} shows the results.
We can say that $\mathrm{Var} [f_C]$ is approximately proportional to $\frac{1}{K}$.
Table \ref{tab:1} shows $D_K(p)$ described as $\mathrm{Var} [f_C]=\frac{D_K(p)}{K}$.
Figure \ref{fig:2} shows comparison between the theoretical values of $\mathrm{Var} [f_C]$ and the corresponding experimentally derived values.
The detail of the experiment is as follows:
\begin{enumerate}
\item Generate $1000$ sequences using Mersenne twister \cite{MT}.
\item Stochastically flip the bits of the sequences following (\ref{eq:flip1}) and (\ref{eq:flip2}) with $\alpha=0.66$.
Random numbers generated by Mersenne twister were used for the stochastic flipping.
\item For fixed $K$, compute $f_C$ for each flipped sequence and get $\mathrm{Var} [f_C]$.
\item Repeat the above operation 30 times and compute the average and the standard deviation of $\mathrm{Var} [f_C]$.
\end{enumerate}
Here, we set $L=4$ and $Q=10\times 2^L$.
As a result, we can confirm that the numerical value based on our theory is valid with high accuracy.

\begin{figure}[htbp]
\centering
\includegraphics[width=10cm]{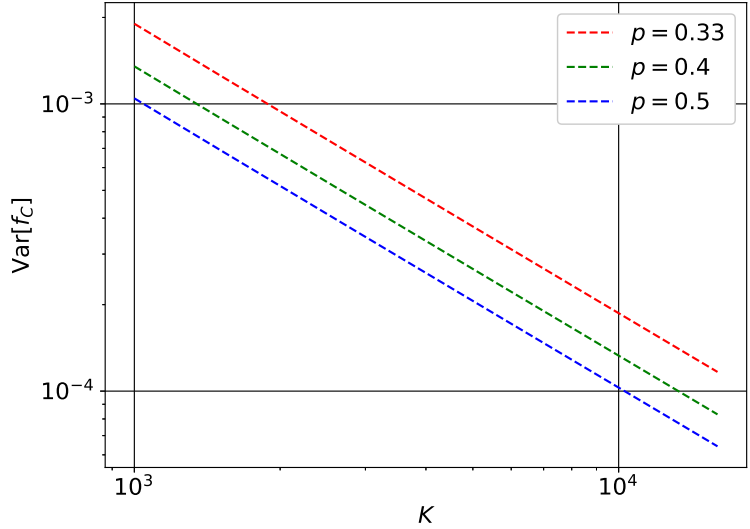}
\caption{$\mathrm{Var} [f_C]$ for $p=0.33,\, 0.4$ and $0.5$ with $L=4$ (Copyright(C)2020 IEICE, \cite{Hikima-Iwasaki-Umeno} Figure 1)}
\label{fig:1}
\end{figure}

\begin{table}[htbp]
  \centering
  \caption{$D_K(p)$ for different values of $p$ and $K$}
  \begin{tabular}{ccccc} \hline
    $p$ & $D_{10000}(p)$ & $D_{20000}(p)$ & $D_{30000}(p)$ & $D_{40000}(p)$  \\ \hline 
    $0.33$    & $1.867364$     & $1.865492$     & $1.864868$     & $1.864556$\\
    $0.4$     & $1.328692$     & $1.327430$     & $1.327009$     & $1.326799$\\
    $0.5$     & $1.028395$     & $1.027449$     & $1.027134$     & $1.026976$\\ \hline
  \end{tabular}
  \label{tab:1}
\end{table}

\begin{figure}[htbp]
\centering
\includegraphics[width=10cm]{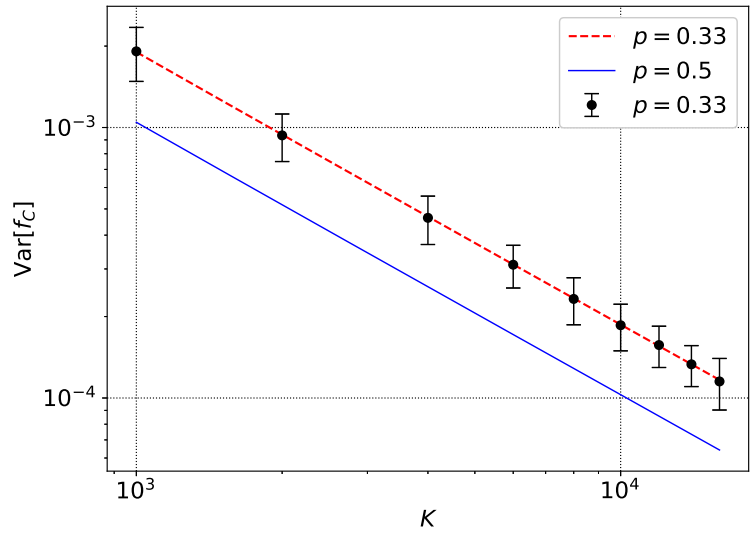}
\caption{Comparison between theoretical and experimental values of $\mathrm{Var} [f_C]$ for $p=0.33$ and $L=4$ (Copyright(C)2020 IEICE, \cite{Hikima-Iwasaki-Umeno} Figure 2)}
\label{fig:2}
\end{figure}

\subsubsection{Experiment 2}\label{subsec:4-3}
Yamamoto et al. suggest that $L=8$ and $K=1000\times 2^8 = 256000$ for their test \cite{Yamamoto-Liu}.
We want to derive $\mathrm{Var} [f_C]$ for such parameters, but directly computing $\mathrm{Var} [f_C]$ requires too much cost.
Then, we computed the exact values of $\mathrm{Var} [f_C]$ for not so large $K$ and approximated the values as
\begin{align}
 \mathrm{Var} [f_C] \simeq\frac{1}{K} \left( \zeta + \frac{\eta}{K} \right),\label{eq:approx_sigma}
\end{align}
where $\zeta$ and $\eta$ are constants.
We estimated the values of $\zeta$ and $\eta$ using the values of $\mathrm{Var} [f_C]$ for $K=40000,45000$, and obtained $\zeta=3.112098237555$ and $\eta=897.7504381251$.
Figure \ref{fig:3} shows a comparison between the exact values of $\mathrm{Var} [f_C]$ and (\ref{eq:approx_sigma}) with such $\zeta$ and $\eta$.
We can say that (\ref{eq:approx_sigma}) approximates $\mathrm{Var} [f_C]$ with high accuracy.
Substituting $K=1000\times 2^8$ into (\ref{eq:approx_sigma}), we obtained
\begin{align}\label{eq:proposed_value}
\mathrm{Var} [f_C] \simeq 1.217033232527091\times10^{-5}
\end{align}
for the suggested $K$.
Figure \ref{fig:4} shows a comparison between (\ref{eq:proposed_value}) and the experimental results using Mersenne twister.
In the experiment, we used $4\times 10^6$ sequences for each trial and set $L=8$, $Q=10\times 2^8$ and $K=1000\times 2^8$.
We can confirm that (\ref{eq:proposed_value}) is consistent with the experimental values and more appropriate than the value used in Ref. \cite{Yamamoto-Liu}.

\begin{figure}[htbp]
\centering
\includegraphics[width=10cm]{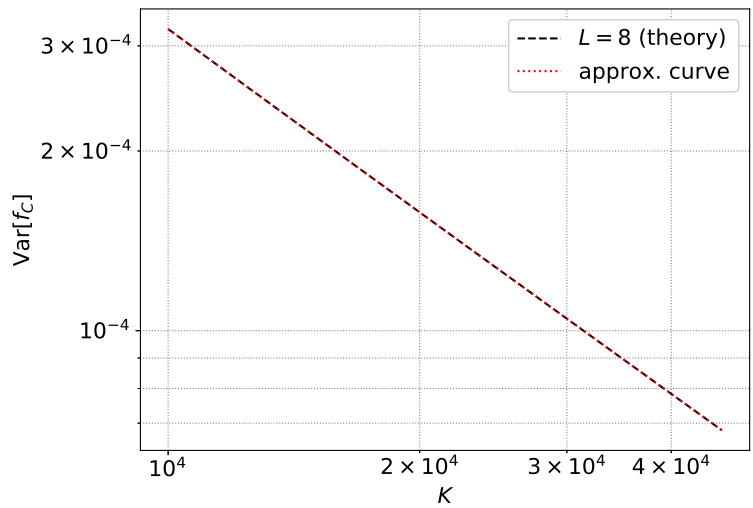}
\caption{Comparison between the theoretical values of $\mathrm{Var} [f_C]$ and the approximation for $p=0.33$ and $L=8$}
\label{fig:3}
\end{figure}

\begin{figure}[htbp]
\centering
\includegraphics[width=10cm]{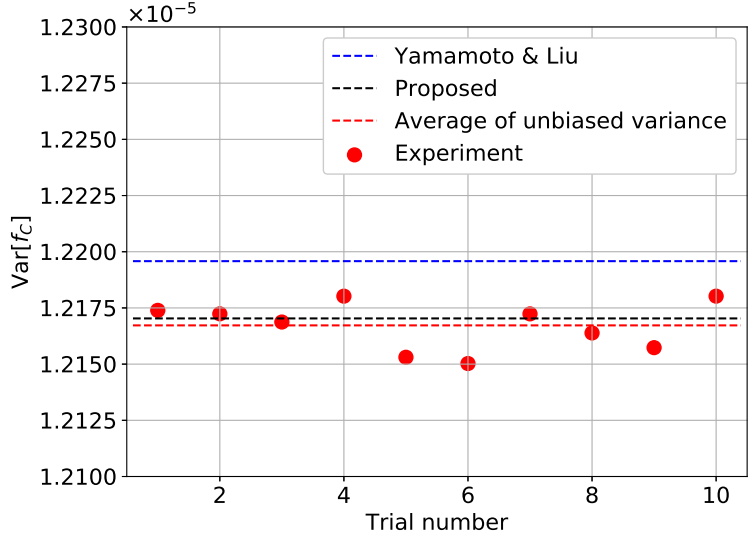}
\caption{Comparison between theoretical and experimental values of $\mathrm{Var} [f_C]$ for $p=0.33$ and $L=8$}
\label{fig:4}
\end{figure}

\section{Form of the reference distribution}
In this section, we prove that the distribution of the test statistic of Coron's test converges to a normal distribution.
More exactly, the test statistic converges to zero as $K\to\infty$ and we prove that the distribution of $\sqrt{K}\times f_C\left(R^{(0.5)}\right)$ converges to a normal distribution in some sense.
This proof can also be applied for Maurer's test and Yamamoto and Liu's test, with minor changes.
In the following, we omit to write argument $R^{(0.5)}$ for every variables.
Note that the omitted argument is different from that in the previous section.
\subsection{Stationary case}
First, we consider the case that $Q\to\infty$, i.e., the random variable sequence $\{A_n\}_{n\in\mathbb{N}}$ satisfies ergodic stationary condition.
We use the following lemmas repeatedly:
\begin{lem}\label{kihonsiki}
Let $M$ be a real number satisfying $0<M<1$.
Then, there exist real numbers $C>0$ and $0<\delta<1$ such that
\begin{align}
\forall p\in\mathbb{N},\ \forall a\in\mathbb{N},\ \forall b\in\mathbb{N},\ 0\leq\sum_{j=a}^{b}\left\{g(j)\right\}^pM^j<C\delta^{a-1}\left(1-\delta^b\right).
\end{align}
\end{lem}
\begin{proof}
Since $g(i)=\mathcal{O}(\log i)$, the lemma is obvious.
\end{proof}
\begin{lem}\label{kihonsiki2}
There exist positive real numbers $C$ and $\delta<1$ such that
\begin{align}
\left|\text{Pr}\left[A_n=i,\ A_{n+k}=j\right]-\text{Pr}\left[A_n=i\right]\text{Pr}\left[A_{n+k}=j\right]\right|<C\delta^{i+j}.
\end{align}
\end{lem}
\begin{proof}
In Ref. \cite{Coron-Naccache}, it is shown that
\begin{align}
\text{Pr}[A_n=i, A_{n+k}=j]
&=\begin{cases}
2^{-2L}(1-2^{-L})^{i+j-2} & (1 \leq j \leq k-1) \\
2^{-2L}(1-2^{-L})^{i+j-2} & (j=k) \\
2^{-2L}(1-2^{-L})^{i-j+2k-1} \left( 1 - 2^{-L+1} \right)^{j-k-1} & (k+1 \leq j \leq k+i-1) \\
0 & (j=k+i) \\
2^{-2L}(1-2^{-L})^{-i+j-1} \left( 1 - 2^{-L+1} \right)^{i-1} & (j \geq k+i+1)
\end{cases},\label{Coron_joint}\\
\text{Pr}\left[A_n=i\right]&=\text{Pr}\left[A_{n+k}=i\right]=2^{-L}(1-2^{-L})^{i-1}.\label{Coron_syuuhen}
\end{align}
(We can get the same results by substituting $p=0.5$ into the results of the former section.)
Then, the lemma holds.
\end{proof}
Let $X_n$ be $g(A_n)$.
\begin{lem}
\label{prop1}
The average $\mathbb{E}\left[X_n\right]$ exists and does not depend on $n$.
\end{lem}
\begin{proof}
Since $\left\{X_n\right\}_{n\in\mathbb{N}}$ is ergodic stationary, it is obvious that $\mathbb{E}\left[X_n\right]$ does not depend on $n$ if it exists.
The existence is directly proven by lemma \ref{kihonsiki} and (\ref{Coron_syuuhen}).
\end{proof}
We use the notation $\mu_X$ as $\mathbb{E}\left[X_n\right]$.
\begin{lem}
\label{prop2}
The variance $\text{Var}\left[\frac{1}{\sqrt{K}}\sum_{n=1}^{K}(X_n-\mu_X)\right]$ exists for all $K\in\mathbb{N}$ and the limit $\lim_{K\to\infty}\text{Var}\left[\frac{1}{\sqrt{K}}\sum_{n=1}^{K}(X_n-\mu_X)\right]$ also exists.
\end{lem}
\begin{proof}
By lemmas \ref{kihonsiki} and \ref{kihonsiki2}, there exist positive real numbers $C_1$, $C_2$ and $\delta<1$ such that
\begin{align}
\text{Var}\left[X_n-\mu_X\right]
<&\sum_{j=1}^\infty\left\{g(j)\right\}^22^{-L}(1-2^{-L})^{j-1}\\
<&C_1,\\
\left|\text{Cov}\left[X_1-\mu_X,X_n-\mu_X\right]\right|
=&\left|\sum_{i=1}^\infty\sum_{j=1}^\infty g(i)g(j) \left(\text{Pr}\left[A_1=i,\ A_n=j\right]-\text{Pr}\left[A_1=i\right]\text{Pr}\left[A_n=j\right]\right)\right|\\
\leq&\sum_{i=1}^\infty\sum_{j=1}^\infty g(i)g(j) \left|\text{Pr}\left[A_1=i,\ A_n=j\right]-\text{Pr}\left[A_1=i\right]\text{Pr}\left[A_n=j\right]\right|\\
=&\sum_{i=1}^\infty\sum_{j=n}^\infty g(i)g(j) \left|\text{Pr}\left[A_1=i,\ A_n=j\right]-\text{Pr}\left[A_1=i\right]\text{Pr}\left[A_n=j\right]\right|\\
<&C_2\delta^n.
\end{align}
Since $\{X_n\}_{n\in\mathbb{N}}$ is ergodic stationary, $\text{Var}\left[X_n-\mu_X\right]$ does not depend on $n$ and \\$\text{Cov}\left[X_n-\mu_X,X_{n+k}-\mu_X\right]$ depends only on $k$.
Then, we obtain
\begin{align}
\text{Var}\left[\frac{1}{\sqrt{K}}\sum_{n=1}^{K}(X_n-\mu_X)\right]
=&\text{Var}\left[X_1-\mu_X\right]+2\sum_{n=2}^{K}\left(1-\frac{n-1}{K}\right)\text{Cov}\left[X_1-\mu_X,X_n-\mu_X\right].
\end{align}
We have
\begin{align}
\sum_{n=2}^{K}\left|\left(1-\frac{n-1}{K}\right)\text{Cov}\left[X_1-\mu_X,X_n-\mu_X\right]\right|
\leq&\sum_{n=2}^{K}\left|\text{Cov}\left[X_1-\mu_X,X_n-\mu_X\right]\right|\\
<&\frac{C_2\delta^2}{1-\delta}\left(1-\delta^{K-1}\right)\ \underset{K\to\infty}{\longrightarrow}\ \frac{C_2\delta^2}{1-\delta}.
\end{align}
From the above, lemma \ref{prop2} holds.
\end{proof}
We use the notation $\tau_X$ as $\lim_{K\to\infty}\text{Var}\left[\frac{1}{\sqrt{K}}\sum_{n=1}^{K}(X_n-\mu_X)\right]$.
For arbitrary $J\in\mathbb{R}$, we describe $Y_n^{(J)}$, $Z_n^{(J)}$ and $M_J$ as
\begin{align}
Y_n^{(J)}:=&\begin{cases}
&X_n\ \ \ \ (X_n\leq J)\\
&\ 0\ \ \ \ \ (X_n> J)
\end{cases}
,\\
Z_n^{(J)}:=&X_n-Y_n^{(J)},\label{teigi:Z}\\
M_J:=&\min \left\{ j\in\mathbb{N}\ |\ g(j)> J\right\}.
\end{align}
Clearly, we have
\begin{align}
\lim_{J\to\infty} M_J=\infty.
\end{align}
\begin{lem}
\label{heikinYZ}
For arbitrary $J\in\mathbb{R}$, $\mathbb{E}\left[Y_n^{(J)}\right]$ and $\mathbb{E}\left[Z_n^{(J)}\right]$ exist and do not depend on $n$.
\end{lem}
\begin{proof}
Since $\left\{Y_n\right\}_{n\in\mathbb{N}}$ is ergodic stationary, $\mathbb{E}\left[Y_n^{(J)}\right]$ does not depend on $n$ if it exists.
For arbitrary $J$ and $n$, the support of $Y_n^{(J)}$ is a finite set.
Therefore, its existence is also apparent.
By lemma \ref{prop1} and (\ref{teigi:Z}), it is clear that $\mathbb{E}\left[Z_n^{(J)}\right]$ exists and does not depend on $n$.
\end{proof}
We use the notation $\mu_Y^{(J)}$ and $\mu_Z^{(J)}$ as $\mathbb{E}\left[Y_n^{(J)}\right]$ and $\mathbb{E}\left[Z_n^{(J)}\right]$, respectively.
Clearly, it is satisfied that
\begin{align}
\label{heikinnowa}
\mu_X=\mu_Y^{(J)}+\mu_Z^{(J)}.
\end{align}
\begin{lem}
For arbitrary $J\in\mathbb{R}$ and $K\in\mathbb{N}$, $\text{Var}\left[\frac{1}{\sqrt{K}}\sum_{n=1}^K\left(Y_n^{(J)}-\mu_Y^{(J)}\right)\right]$\\ and $\text{Var}\left[\frac{1}{\sqrt{K}}\sum_{n=1}^K\left(Z_n^{(J)}-\mu_Z^{(J)}\right)\right]$ exist.
\end{lem}
\begin{proof}
For all $J$, the support of $Y_n^{(J)}-\mu_Y^{(J)}$ is a finite set.
Then, for all $J$ and $K$, the support of $\sum_{n=1}^K\left(Y_n^{(J)}-\mu_Y^{(J)}\right)$ is also a finite set and $\text{Var}\left[\frac{1}{\sqrt{K}}\sum_{n=1}^K\left(Y_n^{(J)}-\mu_Y^{(J)}\right)\right]$ exists.
By (\ref{teigi:Z}) and (\ref{heikinnowa}), we have
\begin{align}
&\text{Var}\left[\frac{1}{\sqrt{K}}\sum_{n=1}^K\left(Z_n^{(J)}-\mu_Z^{(J)}\right)\right]\\
\begin{split}
=& \text{Var}\left[\frac{1}{\sqrt{K}}\sum_{n=1}^K\left(X_k-\mu_X\right)\right]+\text{Var}\left[\frac{1}{\sqrt{K}}\sum_{n=1}^K\left(Y_n^{(J)}-\mu_Y^{(J)}\right)\right]\\
&-2 \text{Cov}\left[\frac{1}{\sqrt{K}}\sum_{n=1}^K\left(X_n-\mu_X\right),\frac{1}{\sqrt{K}}\sum_{n=1}^K\left(Y_n^{(J)}-\mu_Y^{(J)}\right)\right]
\end{split}\\
\begin{split}
\leq& \text{Var}\left[\frac{1}{\sqrt{K}}\sum_{n=1}^K\left(X_n-\mu_X\right)\right]+\text{Var}\left[\frac{1}{\sqrt{K}}\sum_{n=1}^K\left(Y_n^{(J)}-\mu_Y^{(J)}\right)\right]\\
&+2 \sqrt{\text{Var}\left[\frac{1}{\sqrt{K}}\sum_{n=1}^K\left(X_n-\mu_X\right)\right]\text{Var}\left[\frac{1}{\sqrt{K}}\sum_{n=1}^K\left(Y_n^{(J)}-\mu_Y^{(J)}\right)\right]}.
\end{split}
\end{align}
Then, by lemma \ref{prop2}, $\text{Var}\left[\frac{1}{\sqrt{K}}\sum_{k=1}^K\left(Z_k^{(J)}-\mu_Z^{(J)}\right)\right]$ exists for all $J$ and $K$.
\end{proof}
\begin{lem}
\label{lem1}
As $J\to\infty$, $\mu_Z^{(J)}\to0$.
\end{lem}
\begin{proof}
By lemma \ref{kihonsiki} and (\ref{Coron_syuuhen}), there exist positive real numbers $C$ and $\delta<1$ such that
\begin{align}
\left|\mu_Z^{(J)}\right|
=&\left|\sum_{j=M_J}^\infty g(j) 2^{-L}(1-2^{-L})^{j-1}\right|\\
<&C\delta^{M_J}\ \underset{J\to\infty}{\longrightarrow}\ 0.
\end{align}
\end{proof}
\begin{lem}
\label{lem2}
For all $J\in\mathbb{R}$, there exists a positive real number $\tau_Y^{(J)}$ such that $\frac{1}{\sqrt{K}}\sum_{n=1}^{K}(Y_n^{(J)}-\mu_Y^{(J)})$ follows a normal distribution $\mathcal{N}(0,\tau_Y^{(J)})$ as $K\to\infty$.
\end{lem}
\begin{proof}
It is obvious that $\left\{Y_n^{(J)}-\mu_Y^{(J)}\right\}_{n\in\mathbb{N}}$ is ergodic stationary.
By the definitions of $\{Y_n\}_{n\in\mathbb{N}}$, for arbitrary positive integer $l$, $\left\{Y_n^{(J)}-\mu_Y^{(J)}\right\}_{n=1,2,\dots,l}$ and $\left\{Y_n^{(J)}-\mu_Y^{(J)}\right\}_{n=l+M_J,l+M_J+1,\dots}$ are independent each other.
For fixed $J$,
\begin{align}
\mathbb{E}\left[|Y_1^{(J)}-\mu_Y^{(J)}|^3\right]<\infty
\end{align}
because the support of $Y_1^{(J)}$ is a finite set.
Then, by the theorem 2 in Ref. \cite{Hoeffding}, the lemma holds.
\end{proof}
\begin{lem}
\label{lem3}
The following equation is true:
\begin{align}
\lim_{J\to\infty}\lim_{K\to\infty}\text{Var}\left[\frac{1}{\sqrt{K}}\sum_{n=1}^{K}(Z_n^{(J)}-\mu_Z^{(J)})\right]=0.
\end{align}
\end{lem}
\begin{proof}
Since $\left\{Z_n^{(J)}\right\}_{n\in\mathbb{N}}$ is ergodic stationary, we have
\begin{align}
\begin{split}
&\text{Var}\left[\frac{1}{\sqrt{K}}\sum_{n=1}^{K}(Z_n^{(J)}-\mu_Z^{(J)})\right]\\
=&\text{Var}\left[Z_1^{(J)}-\mu_Z^{(J)}\right]+2\sum_{n=2}^{K}\left(1-\frac{n-1}{K}\right)\text{Cov}\left[Z_1^{(J)}-\mu_Z^{(J)},Z_{n}^{(J)}-\mu_Z^{(J)}\right].
\end{split}
\end{align}
By lemma \ref{kihonsiki}, there exist positive real numbers $C_1$ and $\delta_1<1$ such that
\begin{align}
\text{Var}\left[Z_1^{(J)}-\mu_Z^{(J)}\right]
=&\sum_{j=M_J}^\infty\left\{g(j)\right\}^2 2^{-L}\left(1-2^{^L}\right)^{j-1}-\left(\mu_Z^{(J)}\right)^2\\
\leq& \sum_{j=M_J}^\infty C_1\delta_1^{j-1}-\left(\mu_Z^{(J)}\right)^2\\
=&\frac{C_1\delta_1^{M_J}}{1-\delta_1}-\left(\mu_Z^{(J)}\right)^2.
\end{align}
Then, by lemma \ref{lem1},
\begin{align}
\lim_{J\to\infty}\lim_{K\to\infty}\text{Var}\left[Z_1^{(J)}-\mu_Z^{(J)}\right]=0.
\end{align}
By (\ref{Coron_joint}), we have
\begin{align}
\text{Pr}\left[A_1=i,\ A_n=j\right]-\text{Pr}\left[A_1=i\right]\text{Pr}\left[A_n=j\right]=0
\end{align}
for $j\leq n-2$.
Then, by lemmas \ref{kihonsiki} and \ref{kihonsiki2}, for $n\geq M_J+1$, there exist positive real numbers $C_2$ and $\delta_2<1$ such that
\begin{align}
&\left|\text{Cov}\left[Z_1^{(J)}-\mu_Z^{(J)},Z_n^{(J)}-\mu_Z^{(J)}\right]\right|\\
=&\left|\sum_{i=M_J}^\infty\sum_{j=n-1}^\infty g(i) g(j) \left(\text{Pr}\left[A_1=i,\ A_n=j\right]-\text{Pr}\left[A_1=i\right]\text{Pr}\left[A_n=j\right]\right)\right|\\
<&C_2\delta_2^{M_J+n},
\end{align}
and for $n\leq M_J$, there exist positive real numbers $C_3$ and $\delta_3<1$ such that
\begin{align}
&\left|\text{Cov}\left[Z_k^{(J)}-\mu_Z^{(J)},Z_l^{(J)}-\mu_Z^{(J)}\right]\right|\\
=&\left|\sum_{i=M_J}^\infty\sum_{j=M_J}^\infty g(i) g(j) \left(\text{Pr}\left[A_1=i,\ A_n=j\right]-\text{Pr}\left[A_1=i\right]\text{Pr}\left[A_n=j\right]\right)\right|\\
<&C_3\delta_3^{2M_J}.
\end{align}
Then, we obtain
\begin{align}
&\left|\sum_{n=2}^{K}\left(1-\frac{n-1}{K}\right)\text{Cov}\left[Z_1^{(J)}-\mu_Z^{(J)},Z_n^{(J)}-\mu_Z^{(J)}\right]\right|\\
\leq&\sum_{n=2}^{K}\left|\text{Cov}\left[Z_1^{(J)}-\mu_Z^{(J)},Z_n^{(J)}-\mu_Z^{(J)}\right]\right|\\
<&\sum_{n=2}^{M_J}C_3\delta_3^{2M_J}+\sum_{n=M_J+1}^{K}C_2\delta_2^{M_J+n}\\
=&C_3\delta_3^{2M_J}(M_J-1)+\frac{C_2\delta_2^{2M_J+1}}{1-\delta_2}\left(1-\delta_2^{K-M_J}\right)\\
&\underset{K\to\infty}{\longrightarrow}\ C_3\delta_3^{2M_J}(M_J-1)+\frac{C_2\delta_2^{2M_J+1}}{1-\delta_2}\\
&\underset{J\to\infty}{\longrightarrow}\ 0.
\end{align}
From the above, lemma \ref{lem3} is holds.
\end{proof}
\begin{thm}\label{thm_stationary}
As $K\to\infty$, $\frac{1}{\sqrt{K}}\sum_{k=1}^{K}(X_k-\mu_X)$ follows $\mathcal{N}(0,\tau_X)$.
\end{thm}
\begin{proof}
It is proven by lemmas \ref{prop2}, \ref{lem1}, \ref{lem2} and \ref{lem3}.
\end{proof}
\subsection{Nonstationary case}
In real situation, we set $Q$ sufficiently large but finite, and we cannot assume that $\{A_n\}_{n\in\mathbb{N}}$ is ergodic stationary.
However, even in such situation, $\{A_n\}_{n\geq T}$ approximately satisfies ergodic stationary condition as $T\to\infty$ and we can show that using a normal distribution is proper.
Let $X_n^{(Q)}$ be a random variable described as
\begin{align}
X_n^{(Q)}=\begin{cases}
X_n\ &(X_n\leq Q+n-1)\\
g(Q+n-1)\ &(\text{otherwise})
\end{cases}
.
\end{align}
We show that theorem \ref{thm_stationary} holds even if we replace $\{X_n\}_{n\in\mathbb{N}}$ with $\left\{X^{(Q)}_n\right\}_{n\in\mathbb{N}}$.
Concretely, we prove the following.
\begin{thm}\label{thm_nonstationary}
For all $Q$ and $K$, $\mathbb{E}\left[\frac{1}{\sqrt{K}}\sum_{n=1}^K(X_n-X_n^{(Q)})\right]$ and $\text{Var}\left[\frac{1}{\sqrt{K}}\sum_{n=1}^K(X_n-X_n^{(Q)})\right]$ exist, and
\begin{align}
\lim_{K\to\infty}\mathbb{E}\left[\frac{1}{\sqrt{K}}\sum_{n=1}^K(X_n-X_n^{(Q)})\right]&=0,\\
\lim_{K\to\infty}\text{Var}\left[\frac{1}{\sqrt{K}}\sum_{n=1}^K(X_n-X_n^{(Q)})\right]&=0.
\end{align}
\end{thm}
\begin{proof}
By lemma \ref{kihonsiki} and (\ref{Coron_syuuhen}), there exist positive real number $C_1$ and $\delta_1<1$ such that
\begin{align}
0\leq& \mathbb{E}\left[\frac{1}{\sqrt{K}}\sum_{n=1}^K(X_n-X_n^{(Q)})\right]\\
=&\frac{1}{\sqrt{K}}\sum_{n=1}^K\sum_{i=Q+n}^\infty\left(g(i)-g(Q+n-1)\right)\text{Pr}\left[A_n=i\right]\\
<& \frac{1}{\sqrt{K}}\sum_{n=1}^KC_1\delta_1^{Q+n}\\
=&\frac{1}{\sqrt{K}}\frac{C_1\delta_1^{Q+1}}{1-\delta_1}\left(1-\delta_1^K\right)\ \underset{K\to\infty}{\longrightarrow}\ 0.
\end{align}
By lemmas \ref{kihonsiki} and \ref{kihonsiki2}, there exist positive real number $C_2$ and $\delta_2<1$ such that
\begin{align}
&\text{Var}\left[\frac{1}{\sqrt{K}}\sum_{n=1}^K(X_n-X_n^{(Q)})\right]\\
\begin{split}
=&\frac{1}{K}\sum_{n=1}^K\sum_{s=1}^K\sum_{i=Q+n}^\infty\sum_{j=Q+s}^\infty\left\{\left(g(i)-g(Q+n-1)\right)\left(g(j)-g(Q+s-1)\right)\right.\\
&\ \ \ \ \ \ \ \ \ \ \ \ \ \ \ \ \ \ \ \ \ \ \ \ \ \ \ \ \ \ \ \ \left.\times\left(\text{Pr}\left[A_n=i,\ A_s=j\right]-\text{Pr}\left[A_n=i\right]\text{Pr}\left[A_s=j\right]\right)\right\}
\end{split}\\
\begin{split}
\leq&\frac{1}{K}\sum_{n=1}^K\sum_{s=1}^K\sum_{i=Q+n}^\infty\sum_{j=Q+s}^\infty\left\{\left(g(i)-g(Q+n-1)\right)\left(g(j)-g(Q+s-1)\right)\right.\\
&\ \ \ \ \ \ \ \ \ \ \ \ \ \ \ \ \ \ \ \ \ \ \ \ \ \ \ \ \ \ \ \ \left.\times\left|\text{Pr}\left[A_n=i,\ A_s=j\right]-\text{Pr}\left[A_n=i\right]\text{Pr}\left[A_s=j\right]\right|\right\}
\end{split}\\
<&\frac{1}{K}\sum_{n=1}^K\sum_{s=1}^KC_2\delta_2^{2Q+n+s}\\
=&\frac{1}{K}\frac{C_2\delta_2^{2Q+2}}{\left(1-\delta_2\right)^2}\left(1-\delta_2^K\right)^2\ \underset{K\to\infty}{\longrightarrow}\ 0.
\end{align}
Then, theorem \ref{thm_nonstationary} holds.
\end{proof}
\subsection{Proof for Maurer's test and Yamamoto and Liu's test}
The discussion in this section holds even if we replace $g$ with $\log$.
Using the result of the previous section, we can prove lemmas \ref{kihonsiki} even if we replace $R^{(0.5)}$ with $R^{(p)}$ where $p\in(0,1)$.
In addition, for arbitrary $n$, $\text{Pr}\left[A_n(R^{(p)})=i\right]$ exponentially reduce when $i$ become larger, which corresponds to (\ref{Coron_syuuhen}).
Then, the discussion in this section also holds for Maurer's test and Yamamoto and Liu's test, and we can confirm the properness of using normal distributions as their reference distributions.
\section{Conclusion}
We showed the properness that we use normal distributions as the reference distributions of Maurer's test, Coron's test, and Yamamoto and Liu's test.
Additionally, we derived the variance of the reference distribution of Yamamoto and Liu's test, and all of the parameters fixing the normal distributions have been now theoretically derived.
These results let us use these tests properly and evaluate randomness more precisely.

\setcounter{section}{0}
\renewcommand{\appendix}{}
\appendix
\renewcommand{\thesection}{Appendix}
\renewcommand{\thesubsection}{\Alph{section}.\arabic{subsection}.}

\begin{thebibliography}{99}
\bibitem{Tamura-Shikano} K. Tamura and Y. Shikano, ``Quantum Random Numbers generated by the Cloud Superconducting Quantum Computer,'' arXiv:1906.04410 (2019).
\bibitem{TestU01} P. L'Ecuyer and R. Simard, ``TestU01: AC library for empirical testing of random number generators,'' ACM Trans. on Mathematical Software 33.4 (2007): 22.
\bibitem{Diehard} G. Marsaglia, ``DIEHARD: a battery of tests of randomness,'' http://stat.fsu.edu/geo (1996).
\bibitem{Dieharder}G. R. Brown, D. Eddelbuettel, and D. Bauer, ``Dieharder: A random number test suite." Open Source software library, under development, http://www. phy.duke.edu/$\sim$rgb/General/dieharder.php (2013).
\bibitem{NIST} A. Rukhin, at el., `` A Statistical Test Suite for Random and Pseudorandom Number Generators for Cryptographic Applications,'' National Institute of Standards and Technology Special Publication 800-22 revision 1a (2010).
\bibitem{Kim-Umeno-Hasegawa}S. Kim, K. Umeno, and A. Hasegawa, ``On the NIST Statistical Test Suite for Randomness", Technical report of IEICE, ISEC2003-87 (2003).
\bibitem{Hamano}K. Hamano, ``The Distribution of the Spectrum for the Discrete Fourier Transform Test Included in SP800-22", IEICE Trans. Fundamentals, Vol. E88-A, No. 1 (2005).
\bibitem{Pareschi}F. Pareschi, R. Rovatti, and G. Setti, ``On Statistical Test Includeed in the NIST SP800-22 Test Suite and Based on the Binomial Distribution", IEEE trans. Information Forensics and Security, Vol.7, No. 2 (2012).
\bibitem{Okada}H. Okada and K. Umeno, ``Randomness Evaluation with the Discrete Fourier Transform Test Based on Exact Analysis of the Reference Distribution", IEEE Trans. Information Forensics and Security, Vol. 12, No. 5 (2017).
\bibitem{Iwasaki} A. Iwasaki, ``Deriving the Variance of the Discrete Fourier Transform Test Using Parseval's Theorem,'' IEEE Trans. on Information Theory, 66.2 (2020): 1164-1170.
\bibitem{Hamamo-Kaneko} K. Hamano and T. Kaneko, ``Correction of overlapping template matching test included in NIST randomness test suite,'' IEICE transactions on fundamentals of electronics, communications and computer sciences 90.9 (2007): 1788-1792.
\bibitem{Maurer} U. M. Maurer, ``A universal statistical test for random bit generators,'' Journal of cryptology 5.2 (1992): 89-105.
\bibitem{Coron-Naccache} J. S. Coron and D. Naccache, ``An accurate evaluation of Maurer’s universal test,'' International Workshop on Selected Areas in Cryptography. Springer, Berlin, Heidelberg, 1998.
\bibitem{Coron} J. S. Coron, ``On the security of random sources,'' International Workshop on Public Key Cryptography. Springer, Berlin, Heidelberg, 1999.
\bibitem{Yamamoto-Liu} H. Yamamoto and Q. Liu, ``Highly sensitive universal statistical test,'' 2016 IEEE International Symposium on Information Theory (ISIT). IEEE, 2016.
\bibitem{Hikima-Iwasaki-Umeno} Y. Hikima, A. Iwasaki, and K. Umeno, ``The variance of the reference distribution of
highly sensitive universal test constructed on the basis of maurer’s test,'' in Proceedings of
Symposium on Cryptography and Information Security 2020 (SCIS2020), 2A3–4 (in Japanese), IEICE,
2020.
\bibitem{Hikima} Y. Hikima, ``Study on a further improvement of Maurer’s universal statistical test,'' Master's thesis, Department of Applied Mathematics and Physics, Graduate School of Informatics, Kyoto University (2020).
\bibitem{MT}M. Matsumoto and T. Nishimura, ``Mersenne twister: A 623-dimensionally equidistributed uniform pseudorandom number generator", ACM Trans. on Modeling and Computer Simulations, Vol. 8 (1998).
\bibitem{Hoeffding} W. Hoeffding and R. Herbert, ``The central limit theorem for dependent random variables,'' Duke Mathematical Journal 15.3 (1948): 773-780.
\end{thebibliography}
\end{document}